\documentclass[12pt,twoside]{amsart} 
 \title[Characterization of varieties of Fano type]{Characterization of varieties of Fano type via singularities of Cox rings}
\author[Y. Gongyo]{Yoshinori  Gongyo}
\author[S. Okawa]{Shinnosuke  Okawa}
\author[A. Sannai]{Akiyoshi Sannai}
\author[S. Takagi]{Shunsuke  Takagi}
\address{Graduate School of Mathematical Sciences, 
the University of Tokyo, 3-8-1 Komaba, Meguro-ku, Tokyo 153-8914, Japan.}

\email{gongyo@ms.u-tokyo.ac.jp}
\email{okawa@ms.u-tokyo.ac.jp}
\email{sannai@ms.u-tokyo.ac.jp}
\email{stakagi@ms.u-tokyo.ac.jp}

\date{\today, version 1.26}
\dedicatory{Dedicated to Professor~Yujiro~Kawamata on the~occasion of his~sixtieth~birthday.}

\newcommand{\Spec}[0]{{\operatorname{Spec}}}

\newcommand{\Pic}[0]{{\operatorname{Pic}}}

\DeclareMathOperator{\Cl}{Cl}

\usepackage{pxfonts}

\usepackage{latexsym}
\usepackage{amsmath}
\usepackage{amssymb}
\usepackage{amsthm}
\usepackage{amscd}
\usepackage{enumerate} 
\usepackage{amssymb} 
\usepackage{mathrsfs}          
\usepackage[all]{xy}
\usepackage{color}

\usepackage[dvipdfm,bookmarks=true,bookmarksnumbered=true,%
 pdftitle={},%
 pdfsubject={},%
 pdfauthor={},%
 pdfkeywords={TeX; dvipdfmx; hyperref; color;},%
 colorlinks=true]{hyperref}

\newtheorem{thm}{Theorem}[section]

\newtheorem{prop}[thm]{Proposition}

\newtheorem{lem}[thm]{Lemma}

\newtheorem{cor}[thm]{Corollary}

\newtheorem{conj}[thm]{Conjecture}

\newtheorem{cl}[thm]{Claim}
 
\theoremstyle{definition}

\newtheorem{defi}[thm]{Definition}

\newtheorem*{ack}{Acknowledgments}

\theoremstyle{definition}
\newtheorem{rem}[thm]{Remark}

\begin{document}
\bibliographystyle{amsalpha+}
 
 \maketitle
 
\begin{abstract}
We show that every Mori dream space of globally $F$-regular type is of Fano type. 
As an application, we give a characterization of varieties of Fano type in terms of the singularities of their Cox rings.
\end{abstract}

\tableofcontents

\section{Introduction}\label{fano-intro}

The notion of Cox rings was defined in \cite{hk}, generalizing Cox's homogeneous coordinate ring \cite{cox} of projective toric varieties.

Let $X$ be a normal ($\mathbb{Q}$-factorial) projective variety over an algebraically closed field $k$. 
Suppose that the divisor class group $\mathrm{Cl}(X)$ is finitely generated and free, and let $D_1, \cdots, D_r$ be Weil divisors on $X$ which form a basis of $\mathrm{Cl}(X)$. 
Then the ring
$$\bigoplus_{(n_1, \dots, n_r) \in \mathbb{Z}^n }H^0(X, \mathcal{O}_X(n_1D_1+\cdots+n_rD_r)) \subseteq k(X)[t_1^{\pm}, \cdots, t_r^{\pm}]$$
is called the Cox ring of $X$ (for the case when $\Cl{(X)}$ has torsion, see Definition \ref{def of cox ring} and Remark
\ref{ambiguity of Cox ring}). If the Cox ring of a variety $X$ is finitely generated over $k$, $X$ is called a ($\mathbb{Q}$-factorial) Mori dream space. This definition is equivalent to the geometric one given in Definition \ref{Mori dream space}
(\cite[Proposition 2.9]{hk}).
Projective toric varieties are Mori dream spaces and their Cox rings are
isomorphic to polynomial rings \cite{cox}. The converse also holds \cite{hk},
characterizing toric varieties via properties of Cox rings.

We say that $X$ is of \textit{Fano type} if there exists an effective $\mathbb{Q}$-divisor $\Delta$ on $X$ such that $-(K_X+\Delta)$ is ample and $(X, \Delta)$ is klt.  
It is known by \cite{bchm} that $\mathbb{Q}$-factorial varieties of Fano type are Mori dream spaces. 
Since projective toric varieties are of Fano type, this result generalizes
the fact that projective toric varieties are Mori dream spaces.
Therefore, in view of the characterization of toric varieties mentioned above, it is natural to expect
a similar result for varieties of Fano type.
The purpose of this paper is to
give a characterization of varieties of Fano type in terms of the singularities of their Cox rings.

\begin{thm}[{=Theorem \ref{sannai conj}}]\label{charact of fano}
Let $X$ be a $\mathbb{Q}$-factorial
normal projective variety over an algebraically closed field of characteristic zero.
Then $X$ is of Fano type if and only if $X$ is a Mori dream space and its Cox ring has only log terminal singularities.
\end{thm}

Brown \cite{brown-Fano} recently proved that if $X$ is a $\mathbb{Q}$-factorial Fano variety with only log terminal singularities, then its Cox ring has only log terminal singularities via completely different arguments from ours.
Our proof of Theorem \ref{charact of fano} is based on
the notion of global $F$-regularity, which is defined for projective varieties over a field of positive characteristic via splitting of Frobenius morphisms. 
A projective variety over a field of characteristic zero is said to be of globally $F$-regular type if its modulo $p$ reduction is globally $F$-regular for almost all $p$ (see Definition \ref{type} for the precise definition).
Schwede--Smith \cite{schsmith-logfano} proved that varieties of Fano type are of globally $F$-regular type, and they asked whether the converse is true. 
We give an affirmative answer to their question in the case of Mori dream spaces.

\begin{thm}\label{logfano}Let $X$ be a $\mathbb{Q}$-factorial Mori dream space over a field of characteristic zero. Then $X$ is of Fano type if and only if it is of globally $F$-regular type. 
\end{thm}

Theorem \ref{logfano} is a key to the proof of Theorem \ref{charact of fano}, so we outline its proof here. 
The only if part was already proved by \cite[Theorem 5.1]{schsmith-logfano}, so we explain the if part. 
Since $X$ is a $\mathbb{Q}$-factorial Mori dream space, we can run a $(-K_X)$-MMP which terminates in finitely many steps. 
A $(-K_X)$-MMP $X_i \dasharrow X_{i+1}$ usually makes the singularities of $X_i$ worse as $i$ increases, but in our setting, we can check that each $X_i$ is also of globally $F$-regular type. 
This means that each $X_i$ has only log terminal singularities, so that a $(-K_{X})$-minimal model becomes of Fano type. 
Finally we trace back the $(-K_X)$-MMP above and show that in each step the property of being of Fano type is preserved,
concluding the proof.

In order to prove Theorem \ref{charact of fano}, we also show that if $X$ is a $\mathbb{Q}$-factorial Mori dream space of globally $F$-regular type,  then modulo $p$ reduction of a multi-section ring of $X$ is the multi-section ring of modulo $p$ reduction $X_p$ of $X$ for almost all $p$ (Lemma \ref{good model}). The proof is based on the finiteness of contracting rational maps from a
fixed Mori dream space, vanishing theorems for globally $F$-regular varieties and cohomology-and-base-change arguments.
This result enables us to apply the theory of $F$-singularities to a Cox ring of $X$
and, as a consequence, we see that that a $\mathbb{Q}$-factorial Mori dream space over a field of characteristic zero is of globally $F$-regular type if and only if its Cox ring has only log terminal singularities. Thus, Theorem \ref{charact of fano} follows from Theorem \ref{logfano}. 

As an application of Theorem \ref{charact of fano}, we give an alternative proof of \cite[Corollary 3.3]{fg1}, \cite[Corollary 5.2]{fg2} and \cite[Theorem 2.9]{prokshok-mainII}. 

\begin{cor}[{=Theorem \ref{images of logFano}}] 
Let $f:X \to Y$ be a projective morphism between normal projective varieties over an algebraically closed field $k$ of characteristic zero. 
If $X$ is of Fano type, then $Y$ is of Fano type. 
\end{cor}

A normal projective variety $X$ over a field of characteristic zero is said to be of \textit{Calabi--Yau type} if there exists an effective $\mathbb{Q}$-divisor $\Delta$ on $X$ such that $K_X+\Delta \sim_{\mathbb{Q}} 0$ and $(X, \Delta)$ is log canonical, and is said to be of dense globally $F$-split type if its modulo $p$ reduction is Frobenius split for infinitely many $p$. 
Using arguments similar to the proofs of Theorems \ref{charact of fano} and \ref{logfano}, 
we show analogous statements for varieties of Calabi--Yau type. 

\begin{thm}[{=Theorem \ref{CY surface}}]\label{cox ring of 2-dim} 
Let $X$ be a klt projective surface over an algebraically closed field $k$ of characteristic zero such that $K_X \sim_{\mathbb{Q}} 0$. 
If $X$ is a Mori dream space, then its Cox ring has only log canonical singularities. 
\end{thm}

\begin{thm}\label{logcy}
Let $X$ be a $\mathbb{Q}$-factorial Mori dream space over a field of characteristic zero. 
If $X$ is of dense globally $F$-split type, then it is of Calabi--Yau type. 
\end{thm}

\begin{ack}
The first, second, and third author would like to thank Professor Yujiro Kawamata
for his warm encouragement. 
The third author is deeply grateful to Professors Mitsuyasu Hashimoto and Kazuhiko Kurano for their helpful conversations. 
A part of this paper was established during the first, second, and fourth authors were participating in the AGEA conference at National Taiwan University. 
They would like to thank NTU, especially Professor Jungkai Alfred Chen, who was the local organizer of the conference, for their hospitality.

The first and second authors were partially supported by Grant-in-Aid for JSPS Fellows $\sharp$22$\cdot$7399 and $\sharp$22$\cdot$849, respectively. 
The fourth author was partially supported by Grant-in-Aid for Young Scientists (B) 23740024 from JSPS.
\end{ack}

We will freely use the standard notations in \cite{komo} and \cite[3.1. Notation and conventions]{bchm}. 

\section{Preliminaries and lemmas}\label{preliminaries}

\subsection{Mori dream spaces}
Mori Dream Spaces were first introduced by Hu and Keel \cite{hk}. 
\begin{defi}\label{Mori dream space} 
A normal projective variety $X$ over a field is called a \textit{$\mathbb{Q}$-factorial Mori dream space} (or \textit{Mori dream space} for short) 
if $X$ satisfies the following three conditions:
\begin{enumerate}[(i)]
\item $X$ is $\mathbb{Q}$-factorial and $\Pic{(X)}_{\mathbb{Q}} \simeq \mathrm{N}^1{(X)}_{\mathbb{Q}},$\label{fin_pic}
\item $\mathrm{Nef}{(X)}$ is the affine hull of finitely many semi-ample
line bundles, 
\item there exists a finite collection of small birational maps $f_i: X \dasharrow X_i$
such that each $X_i$ satisfies (i) and (ii), and that $\mathrm{Mov}{(X)}$ is the union
of the $f_i^*(\mathrm{Nef}{(X_i)})$.
\end{enumerate}
\end{defi}

On a Mori dream space, as its name suggests, we can run an MMP for any divisor.
\begin{prop}{ (\cite[Proposition 1.11]{hk})}\label{mmp on mds}
Let $X$ be a $\mathbb{Q}$-factorial Mori dream space. 
Then for any divisor $D$ on $X$, a $D$-MMP can be run and terminates. 
\end{prop}

Moreover, we know the finiteness of models for a $\mathbb{Q}$-factorial Mori dream space as follows:

\begin{prop}[{\cite[Proposition 1.11]{hk}}]\label{finiteness of models}
Let $X$ be a $\mathbb{Q}$-factorial Mori dream space. Then there exists finitely many dominant rational contractions $f_i:X \dashrightarrow Y_i$ to a normal projective variety $Y_i$  such that, for any dominant rational contraction $f:X \dashrightarrow Y$, there exists $i$ such that $f \simeq f_i$, i.e. there exists an isomorphism $g:Y \to Y_i$ such that $g \circ f =f_i$ .
\end{prop}

\subsection{Varieties of Fano type and of Calabi--Yau type}
In this paper, we use the following terminology.  

\begin{defi}[{cf.~\cite[Definition 2.34]{komo},\cite[Remark 4.2]{schsmith-logfano}}]\label{sing of pairs}
Let $X$ be a normal variety over a field $k$ of \textit{arbitrary characteristic} and $\Delta$ be an effective $\mathbb{Q}$-divisor on $X$ such that $K_X+\Delta$ is $\mathbb{Q}$-Cartier. 
Let $\pi: \widetilde{X} \to X$ be a birational morphism from a normal variety $\widetilde{X}$.
Then we can write 
$$K_{\widetilde{X}}=\pi^*(K_X+\Delta)+\sum_{E}a(E, X, \Delta) E,$$ 
where $E$ runs through all the distinct prime divisors on $\widetilde{X}$ and the $a(E, X, \Delta)$ are rational numbers. 
We say that the pair $(X, \Delta)$ is \textit{log canonical} (resp. \textit{klt}) if $a(E, X, \Delta) \ge -1$ (resp. $a(E, X, \Delta) >-1$) for every prime divisor $E$ over $X$.  
If $\Delta=0$, we simply say that $X$ has only log canonical singularities (resp. log terminal  singularities). 
\end{defi}

\begin{defi}[cf. {\cite[Lemma-Definition 2.6]{prokshok-mainII}}]\label{Fano pair}Let $X$ be a projective normal variety over a field and $\Delta$ be an effective $\mathbb{Q}$-divisor on $X$ such that $K_X+\Delta$ is $\mathbb{Q}$-Cartier. 
\begin{enumerate}[(i)]
\item We say that $(X,\Delta)$ is a {\em log Fano pair} if $-(K_X+\Delta)$ is ample and $(X, \Delta)$ is klt. 
We say that $X$ is of {\em Fano type} if there exists an effective $\mathbb Q$-divisor $\Delta$ on $X$ such that $(X,\Delta)$ is a log Fano pair. 
\item 
We say that $X$ is of \textit{Calabi--Yau type} if 
there exits an effective $\mathbb{Q}$-divisor $\Delta$ such that $K_X+\Delta \sim_{\mathbb{Q}}0$ and $(X, \Delta)$ is log canonical.  
\end{enumerate}
\end{defi}

\begin{rem}\label{rm-weakfano} If there exists an effective $\mathbb{Q}$-divisor $\Delta$ on $X$ such that $(X,\Delta)$ is klt and $-(K_X+\Delta)$ is nef and big, then $X$ is of Fano type. See \cite[Lemma-Definition 2.6]{prokshok-mainII}.

\end{rem}

\subsection{Globally $F$-regular and $F$-split varieties}

In this subsection, we briefly review the definitions and basic properties of {\em  global $F$-regularity} and {\em global $F$-splitting}. 

A scheme $X$ of prime characteristic $p$ is \textit{$F$-finite} if the Frobenius morphism $F:X \to X$ is finite. 
A ring $R$ of prime characteristic $p$ is called $F$-finite if $\Spec \, R$ is $F$-finite.
For each integer $e \ge 1$, the $e$-th iterated Frobenius pushforward $F^e_*R$ of a ring $R$ of prime characteristic $p$ is $R$ endowed with an $R$-module structure given by the $e$-th iterated Frobenius map $F^e:R \to R$. 

\begin{defi}\label{defi-sFreg}
Let $R$ be an $F$-finite integral domain of characteristic $p>0$. 

\begin{enumerate}[(i)]
\item 
We say that $R$ is \textit{$F$-pure} if the Frobenius map 
$$F:R \to F_*R \quad a \to a^p$$
splits as an $R$-module homomorphism.

\item
We say that $R$ is \textit{strongly $F$-regular} if for every nonzero element $c \in R$, there exists an integer $e \ge 1$ such that 
$$cF^e:R \to F^e_*R \quad a \to ca^{p^e}$$ 
splits as an $R$-module homomorphism. 

\end{enumerate}
An $F$-finite integral scheme $X$ has only $F$-pure (resp. strongly $F$-regular) singularities if $\mathcal{O}_{X, x}$ is $F$-pure (resp. strongly $F$-regular) for all $x \in X$. 
\end{defi}

\begin{defi}\label{defi-gFreg} 
Let $X$ be a normal projective variety defined over an $F$-finite field of characteristic $p>0$. 
\begin{enumerate}[(i)]
\item 
We say that $X$ is \textit{globally $F$-split} if the Frobenius map
$$\mathcal{O}_X \to F_*\mathcal{O}_X$$
splits as an $\mathcal{O}_X$-module homomorphism. 
\item
We say that $X$ is \textit{globally $F$-regular} if for every effective divisor $D$ on $X$, there exists an integer $e \ge 1$ such that the composition map 
$$\mathcal{O}_X \to F^e_*\mathcal{O}_X \hookrightarrow F^e_*\mathcal{O}_X(D)$$
of the $e$-times iterated Frobenius map $\mathcal{O}_X \to F^e_*\mathcal{O}_X$ with a natural inclusion $F^e_*\mathcal{O}_X \hookrightarrow F^e_*\mathcal{O}_X(D)$
splits as an $\mathcal{O}_X$-module homomorphism. 
\end{enumerate}
\end{defi}

\begin{rem}
Globally $F$-regular (resp. globally $F$-split) varieties have only strongly $F$-regular (resp. $F$-pure) singularities. 
\end{rem}

Let $X$ be a normal projective variety over a field. For any ample Cartier divisor $H$ on $X$, we denote the corresponding \textit{section ring} by
$$R(X, H)=\bigoplus_{m \ge 0}H^0(X, \mathcal{O}_X(mH)).$$

\begin{prop}[{\cite[Proposition 3.1 and Theorem 3.10]{smith}}]\label{glsec}
Let $X$ be a normal projective variety over an $F$-finite field of characteristic $p>0$. The following conditions are equivalent to each other:
\begin{enumerate}[$(1)$]
\item $X$ is globally $F$-split $($resp. globally $F$-regular$)$, 
\item the section ring $R(X, H)$ with respect to some ample divisor $H$ is $F$-pure $($resp. strongly $F$-regular$)$, 
\item the section ring $R(X, H)$ with respect to every ample divisor $H$ is $F$-pure $($resp. strongly $F$-regular$)$. 
\end{enumerate}
\end{prop}

\begin{thm}[{\cite[Theorem 4.3]{schsmith-logfano}}]\label{ss1} 
Let $X$ be a normal projective variety over an $F$-finite field of characteristic $p>0$.  
If $X$ is globally $F$-regular $($resp. globally $F$-split$)$, then $X$ is of Fano type $($resp. Calabi--Yau type$)$.  
\end{thm}

\begin{lem}\label{glFness for MMP}Let $f: X \dashrightarrow X_1$ be a small birational map or an  algebraic fiber space of normal varieties over an $F$-finite field of characteristic $p>0$. 
If $X$ is globally $F$-regular (resp. globally $F$-split), then so is $X_1$.
\end{lem}

\begin{proof}\label{cl1} 
When $f$ is an algebraic fiber space,  the globally $F$-split case follows from \cite[Proposition 4]{mr}  and the globally $F$-regular case does from \cite[Proposition 1.2 (2)]{HWY}. 

When $f$ is a small birational map, $X$ and $X_1$ are isomorphic in codimension one. 
In general, a normal projective variety $Y$ is globally $F$-regular (resp. globally $F$-split) if and only if so is $Y \setminus E$, where $E \subseteq Y$ is a closed subset of codimension at least two (see \cite[1.1.7 Lemma (iii)]{BK} for the globally $F$-split case and \cite[Lemma 2.9]{Hashi} for globally $F$-regular case).
Thus, we obtain the assertion.   
\end{proof}

Now we briefly explain how to reduce things from characteristic zero to characteristic $p > 0$. 
The reader is referred to \cite[Chapter 2]{HH} and \cite[Section 3.2]{MS} for details. 

Let $X$ be a normal variety over a field $k$ of characteristic zero and $D=\sum_i d_i D_i$ be a $\mathbb{Q}$-divisor on $X$. 
Choosing a suitable finitely generated $\mathbb{Z}$-subalgebra $A$ of $k$, 
we can construct a scheme $X_A$ of finite type over $A$ and closed subschemes $D_{i, A} \subsetneq X_A$ such that 
there exists isomorphisms 
\[\xymatrix{
X \ar[r]^{\cong \hspace*{3em}} &  X_A \times_{\Spec \, A} \Spec \, k\\
D_i \ar[r]^{\cong \hspace*{3em}} \ar@{^{(}->}[u] & D_{i, A} \times_{\Spec \, A} \Spec \, k. \ar@{^{(}->}[u]\\
}\]
Note that we can enlarge $A$ by localizing at a single nonzero element and replacing $X_A$ and $D_{i,A}$ with the corresponding open subschemes. 
Thus, applying the generic freeness \cite[(2.1.4)]{HH}, we may assume that $X_A$ and $D_{i, A}$ are flat over $\Spec \, A$.
Enlarging $A$ if necessary, we may also assume that $X_A$ is normal and $D_{i, A}$ is a prime divisor on $X_A$. 
Letting $D_A:=\sum_i d_i D_{i,A}$, we refer to $(X_A, D_A)$ as a \textit{model} of $(X, D)$ over $A$.   

Given a closed point $\mu \in \Spec \, A$, we denote by $X_{\mu}$ (resp., $D_{i, \mu}$) the fiber of $X_A$ (resp., $D_{i, A}$) over $\mu$.  
Then $X_{\mu}$ is a scheme of finite type over the residue field $k(\mu)$ of $\mu$, which is a finite field.  
Enlarging $A$ if necessary, we may assume that  $X_{\mu}$ is a normal variety over $k(\mu)$, $D_{i, \mu}$ is a prime divisor on $X_{\mu}$ and consequently $D_{\mu}:=\sum_i d_i D_{i, \mu}$ is a $\mathbb{Q}$-divisor on $X_{\mu}$ for all closed points $\mu \in \Spec \, A$. 

Let $\Gamma$ be a finitely generated group of Weil divisors on $X$. 
We then refer to a group $\Gamma_A$ of Weil divisors on $X_A$ generated by a model of a system of generators of $\Gamma$ as a \textit{model} of $\Gamma$ over $A$. 
After enlarging $A$ if necessary, we denote by $\Gamma_{\mu}$ the group of Weil divisors on $X_{\mu}$ obtained by restricting divisors in $\Gamma_A$ over $\mu$. 

Given a morphism $f:X \to Y$ of varieties over $k$ and a model $(X_A, Y_A)$ of $(X, Y)$ over $A$,  after possibly enlarging $A$, we may assume that $f$ is induced by a morphism $f_A :X_A \to Y_A$ of schemes of finite type over $A$. 
Given a closed point $\mu \in \Spec \, A$, we obtain a corresponding morphism $f_{\mu}:X_{\mu} \to Y_{\mu}$ of schemes of finite type over $k(\mu)$. 
If $f$ is projective (resp. finite), after possibly enlarging $A$, we may assume that $f_{\mu}$ is projective (resp. finite) for all closed points $\mu \in \Spec \, A$. 

\begin{defi}\label{type}
Let the notation be as above. 
\begin{enumerate}[(i)]
\item A projective variety (resp. an affine variety) $X$ is said to be of \textit{globally $F$-regular type} (resp. \textit{strongly $F$-regular type}) if for a model of $X$ over a finitely generated $\mathbb{Z}$-subalgebra $A$ of $k$, there exists a dense open subset $S \subseteq \Spec \, A$ of closed points such that $X_{\mu}$ is globally $F$-regular (resp. strongly $F$-regular) for all $\mu \in S$. 
\item A projective variety (resp. an affine variety) $X$ is said to be of \textit{dense globally $F$-split type} (resp. \textit{dense $F$-pure type}) if for a model of $X$ over a finitely generated $\mathbb{Z}$-subalgebra $A$ of $k$, there exists a dense subset $S \subseteq \Spec \, A$ of closed points such that $X_{\mu}$ is globally $F$-split (resp. $F$-pure) for all $\mu \in S$. 
\end{enumerate}
\end{defi}

\begin{rem}
(1) The above definition is independent of the choice of a model. 

(2)
If $X$ is of globally $F$-regular type (resp. strongly $F$-regular type), then we can take a model $X_A$ of $X$ over some $A$ such that $X_{\mu}$ is globally $F$-regular (resp. strongly $F$-regular) for all closed points $\mu \in \Spec \, A$. 
\end{rem}

\begin{prop}\label{just singularities}
Let $X$ be a normal projective variety over a field of characteristic zero.
\begin{enumerate}[$(1)$]
\item If $X$ is $\mathbb{Q}$-Gorenstein and of globally $F$-regular type $($resp. dense globally $F$-split type$)$, then it has only log terminal singularities $($resp. log canonical singularities$)$. 
\item If $X$ is of Fano type, then $X$ is of globally $F$-regular type.  
\end{enumerate}
\end{prop}
\begin{proof}
(2) is nothing but \cite[Theorem 5.1]{schsmith-logfano}. So, we will prove only (1). 

Since $X$ is of globally $F$-regular type (resp. dense globally $F$-regular type), then it has only singularities of strongly $F$-regular type (resp. dense $F$-pure type). 
It then follows from \cite[Theorem 3.9]{hw} that $X$ has only log terminal singularities (resp. log canonical singularities). 
\end{proof}

\subsection{Cox rings and their reductions to positive characteristic}
In this paper, we define Cox rings as follows: 

\begin{defi}[Multi-section rings and Cox rings]\label{def of cox ring}
Let $X$ be an integral normal scheme. For a semi-group $\Gamma$ of Weil divisors on $X$,
the $\Gamma$-graded ring
\begin{equation*}
R_X(\Gamma)=\bigoplus_{D\in\Gamma}H^0(X,\mathcal{O}_X(D))
\end{equation*}
is called the \textit{multi-section ring} of $\Gamma$.

Suppose that $\Cl{(X)}$ is finitely generated.
For such $X$, choose a group $\Gamma$ of Weil divisors on $X$ such that
$\Gamma_{\mathbb{Q}}\to \Cl{(X)}_{\mathbb{Q}}$ is an isomorphism.
Then the multi-section ring $R_X{(\Gamma)}$ is called a \textit{Cox ring} of $X$.
\end{defi}

\begin{rem}\label{ambiguity of Cox ring}
As seen above, the definition of a Cox ring depends on a choice of the group $\Gamma$.
When $\Cl{(X)}$ is a free group, it is common to take $\Gamma$ so that the natural map
$\Gamma\to\Cl{(X)}$ is an isomorphism. In this case, the corresponding multi-section ring does not
depend on the choice of such a group $\Gamma$, up to isomorphisms. 
In general Cox rings are not unique. 
Here we note that the basic properties of Cox rings are not affected by the ambiguity.

Let $m$ be a positive integer. Then the natural inclusion
$R_X(m\Gamma)\subset R_X(\Gamma)$ is an integral extension. Therefore
$R_X(\Gamma)$ is of finite type if $R_X(m\Gamma)$ is. Conversely,
we can represent $R_X(m\Gamma)$ as an invariant subring of $R_X(\Gamma)$
under an action of a finite group scheme. Therefore $R_X(m\Gamma)$ is of finite type
if $R_X(\Gamma)$ is. This shows that the finite generation of a Cox ring does not depend
on the choice of $\Gamma$.

Suppose that $m$ is not divisible by the characteristic of the base field.
Then $R_X(m\Gamma)\subset R_X(\Gamma)$ is \'{e}tale in codimension one
(this follows from \cite[Lemma 5.7.(1)]{schsmith-logfano}. See also \cite[Lemma 5.2.]{brown-Fano}).
This shows that in characteristic zero the log-canonicity (resp. log-terminality) of a Cox ring does not depend on the
choice of $\Gamma$, provided that they are of finite type (\cite[Proposition 5.20]{komo}).

Finally, $F$-purity (resp. quasi-$F$-regularity) of a Cox ring is also independent of the choice of $\Gamma$.
We prove it for $F$-purity, and the arguments for quasi-$F$-regularity are the same.
Suppose that $R_X(\Gamma)$ is a Cox ring and is $F$-pure. Take an ample divisor $H\in\Gamma$. Then
$R(X, H)=R_X(\mathbb{N}H)$ is also $F$-pure, since
$\mathbb{N}H$ is a sub-semigroup of $\Gamma$ (use the argument in the proof of Lemma \ref{lem-ring-sum} below).
By Proposition \ref{glsec}, this implies that $X$ is globally $F$-split. By Lemma \ref{sannai lem},
the multi-section ring $R_X(\Gamma')$ of any semigroup $\Gamma'$ of Weil divisors on $X$ is
$F$-split.
\end{rem}

The following is a basic fact on the finite generation of Cox rings. 
\begin{rem}[{\cite[Proposition 2.9]{hk}}]
Let $X$ be a normal projective variety satisfying $(\textup{\ref{fin_pic}})$ of Definition \ref{Mori dream space}.
Then $X$ is a Mori dream space if and only if its Cox rings are finitely generated over $k$.
\end{rem}

If the variety is a $\mathbb{Q}$-factorial Mori dream space of globally $F$-regular type, then we can show that taking multi-section rings commutes with reduction modulo $p$. 

\begin{lem}\label{good model}
Let $X$ be a $\mathbb{Q}$-factorial Mori dream space defined over a field $k$ of characteristic zero and $\Gamma$ 
be a finitely generated group of Cartier divisors on $X$. 
Suppose that $X$ is of globally $F$-regular type $($resp. dense globally $F$-split type$)$.
Then, replacing $\Gamma$ with a suitable positive multiple if necessary, we can take
a model $(X_A, \Gamma_A)$ of $(X, \Gamma)$ over a finitely generated $\mathbb{Z}$-subalgebra $A$ of $k$ and a dense open subset $($resp. a dense subset$)$ $S \subseteq \Spec \, A$ of closed points such that
\begin{enumerate}
\item $X_{\mu}$ is globally $F$-regular $($resp. globally $F$-split$)$, 
\item one has 
\begin{equation*}
(R_X(\Gamma))_{\mu}=R_{X_A}(\Gamma_A)\otimes_A k(\mu)\simeq R_{X_{\mu}}(\Gamma_{\mu})
\end{equation*}
\end{enumerate}
for every $\mu\in S$.
\end{lem}

\begin{proof}
We will show that there exists an integer $m \ge 1$, a model $(X_A, \Gamma_A)$ of $(X, \Gamma)$ over a finitely generated $\mathbb{Z}$-subalgebra $A$ of $k$ and a dense open subset (resp. a dense subset) $S \subseteq \Spec \, A$ of closed points such that 
for every $\mu \in S$ and every divisor $D_A \in m\Gamma_A$, 
\begin{enumerate}
\item $X_{\mu}$ is globally $F$-regular (resp. globally $F$-split), 
\item one has 
\begin{equation*}
H^0(X_A, \mathcal{O}_{X_{A}}(D_A))\otimes_A k(\mu)\simeq H^0(X_{\mu}, \mathcal{O}_{X_{\mu}}(D_{\mu})). 
\end{equation*}
\end{enumerate}

First note that for every divisor $D \in \Gamma$, a $D$-MMP can be run and terminates by Lemma \ref{mmp on mds}. 
It follows from Proposition \ref{finiteness of models} that there exist finitely many birational contractions $f_i: X \dashrightarrow Y_i$ and finitely many projective morphisms $g_{i j}: Y_i \to Z_{ij}$ with connected fibers, where the $Y_i$ are $\mathbb{Q}$-factorial Mori dream spaces and the $Z_{ij}$ are normal projective varieties,  satisfying the following property: 
for every divisor $D \in \Gamma$, there exist $i$ and $j$ such that $f_i: X \dashrightarrow Y_i$ is isomorphic to a composition of $D$-flips and $D$-divisorial contractions and $g_{ij}: Y_i \to Z_{ij}$ is the $D$-canonical model or the $D$-Mori fiber space. 
Then by Lemma \ref{glFness for MMP}, all $Y_i$ are of globally $F$-regular type (resp. dense globally $F$-split type).  

Suppose given models $f_{i, A}: X_A \dashrightarrow Y_{i, A}$ and $g_{ij, A}: Y_{i, A} \to Z_{ij, A}$ of the $f_i$ and $g_{ij}$ over a finitely generated $\mathbb{Z}$-subalgebra $A$ of $k$, respectively. 
Enlarging $A$ if necessary, we may assume that $Y_{i, \mu}, Z_{ij, \mu}$ are normal varieties for all closed points $\mu \in \Spec \, A$. 
In addition, after possibly enlarging $A$ again, we can assume that all $f_{i, A}$ and $f_{i, \mu}$ are compositions of small maps and divisorial birational contractions, and all $g_{ij, A}$ and $g_{ij,\mu}$ are algebraic fiber spaces for all closed points $\mu \in \Spec\, {A}$.

By Definition \ref{Mori dream space} (ii), we take a sufficiently large $m$ so that for each pseudo-effective divisor $D \in \Gamma$, there exists some $i, j$ and a very ample Cartier divisor $H_{mD}$ on $Z_{ij}$ such that $m{f_i}_*D \sim g^{*}_{ij}H_{mD}$. 
By Definition \ref{Mori dream space} (ii) again, enlarging $A$ if necessary, we may assume that models $H_{mD, A}$ of the $H_{mD}$ are given over $A$. 

Now we fix an effective divisor $D\in\Gamma$, and choose $f_i$, $g_{ij}$ for this divisor $D$ as above. 
Then
\begin{align*}
H^0(X_A,\mathcal{O}_{X_A}(mD_A))& \simeq H^0(Y_{i,A}, \mathcal{O}_{Y_{i,A}}(mf_{i,A*}D_A))\\
& \simeq H^0(Z_{ij,A}, \mathcal{O}_{Z_{ij,A}}(H_{mD,A}))
\end{align*}
Similarly, for all closed points $\mu\in\Spec\,{A}$, we have
$$H^0(X_{\mu},\mathcal{O}_{X_{\mu}}(mD_{\mu}))\simeq H^0(Z_{ij, \mu}, \mathcal{O}_{Z_{ij,\mu}}(H_{mD,\mu})).$$
We then use the following claim.

\begin{cl}\label{okawa lem}
Let $W_A$ be a normal projective variety over a finitely generated $\mathbb{Z}$-algebra $A$ such that $W_{\mu}:=W_A \otimes_A k(\mu)$ is globally $F$-split for all closed points $\mu$ in a dense subset $S$ of $\Spec\,A$, and let $H_A$ be an ample Cartier divisor on $W_A$. 
Then
 $$H^0(W_{A},\mathcal{O}_{W_A}(H_A))\otimes_{A}k(\mu)\simeq
H^0(W_{\mu}, \mathcal{O}_{W_{\mu}}(H_{\mu})).$$ 
for all closed points $\mu \in S$. 
\end{cl}

\begin{proof}[Proof of Claim \ref{okawa lem}]

First note that 
$H^1(W_{\mu},\mathcal{O}_{W_{\mu}}(H_{\mu}))=0$ for all closed points $\mu \in S$.  
In fact, since $H_{\mu}$ is ample and $X_{\mu}$ is globally $F$-split, this follows from \cite[Proposition 3]{mr}. 
By \cite[Chapter III, Corollary 12.9]{har}, we see that
$H^1(W_A,\mathcal{O}_{W_A}(H_A))=0$. Applying \cite[Chapter III, Theorem 12.11 (b)]{har} for $i=1$, and then
\cite[Chapter III, Theorem 12.11 (a)]{har} for $i=0$, we get the conclusion.
\end{proof}

Applying the above claim to $Z_{ij, A}$ and $H_{mD, A}$, we see that
$$H^0(X_A, \mathcal{O}_{X_{A}}(mD_A))\otimes_A k(\mu)\simeq H^0(X_{\mu}, \mathcal{O}_{X_{\mu}}(mD_{\mu}))$$
for all effective divisors $D_A \in \Gamma_A $ and all closed points $\mu \in S$.

Next we consider the case when a divisor $D \in \Gamma$ is not effective. 
In particular, $H^0(X,\mathcal{O}_X(D))=0$. 
Choose $f_i$, $g_{ij}$ for this divisor $D$ as above, 
and we take a $g_{ij}$-contracting curve $C_{ij}$. 
Note that the class $[C_{ij}] \in \mathrm{NE}^1(Y_{i})$ is a movable class. 
After possibly enlarging $A$, we can take a model $C_{ij,A}$ of $C_{ij}$ over $A$ such that the class $[C_{ij,\mu}] \in \mathrm{NE}^1(Y_{i,\mu})$ is also a movable class for all closed points $\mu \in \Spec\,A$. Then
$$f_{i,\mu*}D_{\mu}.C_{i,j,\mu}=f_{i*}D.C_{i,j}<0,$$
which implies
$$H^0(X_{\mu}, \mathcal{O}_{X_{\mu}}(mD_\mu))\simeq H^0(Y_{i,\mu}, \mathcal{O}_{Y_{i,\mu}}(mf_{i,\mu*}D_{\mu}))=0$$
for all closed points $\mu \in \Spec\,{A}$.
Thus, 
$$
H^0(X_A, \mathcal{O}_{X_A}(mD_A))\otimes_{A}k(\mu)=H^0(X_\mu, \mathcal{O}_{X_{\mu}}(mD_{\mu}))
$$
 holds for all divisors $D_A \in \Gamma_A$ and all closed points $\mu\in\Spec\,{A}$. 
\end{proof}

\section{Proofs of Theorems \ref{logfano} and \ref{logcy} 
}\label{proof of main} 

In this section, we give proofs of Theorems \ref{logfano} and \ref{logcy}.

\subsection{Globally $F$-regular case}

The following lemma is a special case of \cite[Corollary 3.3]{fg1} and \cite[Theorem 2.9]{prokshok-mainII}, which follow from Kawamata's semi-positivity theorem and Ambro's canonical bundle formula (cf. \cite{ambro-canonical}), respectively. 
We, however, do not need any semi-positivity type theorem for the proof of Lemma \ref{inv-flip}. 

\begin{lem}[cf. {\cite[Theorem\,3.1]{fg1}}]\label{inv-flip}Let $X$ be a normal variety over a field of characteristic zero and $f:X \to Y $ be a small projective birational contraction. Then  $X$ is of Fano type if and only if so is $Y$.
\end{lem}

\begin{proof}First we assume that $X$ is of Fano type, that is, there exists an effective $\mathbb{Q}$-divisor $\Delta$ on $X$ such that $(X,\Delta)$ is a log Fano pair. 
Let $H$ be a general ample divisor on Y, and take a sufficiently small rational number $\epsilon>0$ so that $-(K_X+\Delta+\epsilon f^*H)$ is ample and $(X, \Delta+\epsilon f^*H)$ is klt. 
We also take a general effective ample $\mathbb{Q}$-divisor $A$ on $X$ such that $(X, \Delta+\epsilon f^*H+A)$ is klt and 
$$K_X+\Delta+\epsilon f^*H+A \sim_{\mathbb{Q}}0.
$$
Then 
$$K_Y+f_*\Delta+\epsilon H +f_*A=f_*(K_X+\Delta+\epsilon f^*H+A) \sim_{\mathbb{Q}} 0.$$
On the other hand, since $f$ is small, 
$$f^*(K_Y+f_*\Delta+\epsilon H +f_*A)=K_X+\Delta+\epsilon f^*H+A$$
Therefore, $(Y, f_*\Delta+f_*A)$ is klt and $-(K_Y+f_*\Delta+f_*A) \sim_{\mathbb{Q}} \epsilon H$, which means that $Y$ is of Fano type.
 
 Conversely, we assume that $Y$ is of Fano type. Let $\Gamma$ be an effective $\mathbb{Q}$-divisor on $Y$ such that $(Y,\Gamma)$ is a log Fano pair and let $\Gamma_X$ denote the strict  transform of $\Gamma$ on $X$. Since $f$ is small, we see that
 $$K_X+\Gamma_X=f^*(K_Y+\Gamma).
 $$  
 Thus, $(X,\Gamma_X)$ is klt and $-(K_X+\Gamma_X)$ is nef and big. 
 It then follows from Remark \ref{rm-weakfano} that $X$ is of Fano type. 
\end{proof}

\begin{proof}[Proof of Theorem \ref{logfano}]\label{proof of mthm}
The only if part follows from Proposition \ref{just singularities} (2), so we will prove the if part. 

First we remark that if $Y$ is a $\mathbb{Q}$-factorial Mori dream space of globally $F$-regular type, then $-K_Y$ is big. Indeed, choosing a suitable integer $m \ge 1$, by Lemma \ref{good model}, we can take a model $Y_A$ of $Y$ over a finitely generated $\mathbb{Z}$-subalgebra $A$ such that 
\begin{enumerate}
\item $Y_{\mu}=Y_A \times_{\Spec \, A} \Spec \, k(\mu)$ is globally $F$-regular,
\item $R(Y, -mK_Y)_{\mu} \cong R(Y_{\mu},  -mK_{Y_{\mu}})$
\end{enumerate}
for all closed points $\mu \in \Spec \, A$. 
Since $-K_{Y_{\mu}}$ is big by (1) and Theorem \ref{ss1}, it follows from (2) that $-K_Y$ is also big. 

Since $X$ is a $\mathbb{Q}$-factorial Mori dream space, we can run a $(-K_X)$-MMP: 
$$X=X_0 \overset{f_0}{\dashrightarrow} X_1 \overset{f_1}{\dashrightarrow} \cdots  \overset{f_{l-2}}{\dashrightarrow} X_{l-1} \overset{f_{l-1}}{\dashrightarrow} X_l=X',$$ 
where each $X_i$ is a $\mathbb{Q}$-factorial Mori dream space and $X'$ is a $(-K_X)$-minimal model. 
Note that each $X_i$ is of globally $F$-regular type by Lemma \ref{glFness for MMP}. 
In particular, $-K_{X'}$ is nef and big, and $X'$ has only log terminal singularities by Lemma \ref{just singularities} (1). 
It then follows from Remark \ref{rm-weakfano} that $X'$ is of Fano type. 

Now we show that $X_{l-j}$ is of Fano type by induction on $j$. 
When $j=0$, we have already seen that $X'=X_l$ is of Fano type. 
Suppose that $X_{l-j+1}$ is of Fano type. 
Let $\Delta_{l-j+1}$ be an effective $\mathbb{Q}$-divisor on $X_{l-j+1}$ such that $(X_{
l-j+1},\Delta_{l-j+1})$ is a log Fano pair.

 When $f:=f_{l-j}$ is a divisorial contraction, $K_{X_{l-j}}$ is $f$-ample and $\Delta_{l-j}:=f_{*}^{-1}\Delta_{l-j+1}$ is $f$-nef. 
In particular, $K_{X_{l-j}}+\Delta_{l-j}$ is $f$-ample.
It then follows from the negativity lemma that  
$$-(K_{X_{l-j}}+\Delta_{l-j})= -f^*(K_{X_{l-j+1}}+ \Delta_{l-j+1})+a E,
$$ 
where $a$ is a positive rational number and $E$ is the $f$-exceptional prime divisor on $X_{l-j}$. 
We see from this that the pair $(X_{l-j}, \Delta_{l-j}+aE)$ is klt and $-(K_{X_{l-j}}+\Delta_{l-j}+aE)$ is nef and big, which implies  by Remark \ref{rm-weakfano} that $X_{l-j}$ is of Fano type.

When $f_{l-j}$ is a $(-K_{X_{l-j}})$-flip, we consider the following flipping diagram:
$$
\xymatrix{
X_{l-j} \ar@{-->}[rr]^{f_{l-j}} \ar[dr]_{\psi_{l-j}} & & X_{l-j+1} \ar[dl]^{\psi_{l-j}^{+}}\\
& Z_{l-j} & 
}
$$
Applying Lemma \ref{inv-flip} to $\psi_{l-j}$ and $\psi_{l-j}^{+}$, we see that  $X_{l-j}$ is of Fano type. 

Thus, we conclude that $X=X_0$ is of Fano type.  
\end{proof}

\subsection{Globally $F$-split case}
In this subsection, we start with the following lemma. 
An analogous statement for klt Calabi--Yau pairs follows from \cite[Theorem 0.2]{ambro-canonical}, but our proof of Lemma \ref{inv-flip-cy} is easier. 

\begin{lem}\label{inv-flip-cy}Let $X$ be a normal variety over a field of characteristic zero and $f:X \to Y $ be a small projective birational contraction. Then  $X$ is of Calabi--Yau type if and only if so is $Y$.
\end{lem}

\begin{proof} 
Suppose that $X$ is of Calabi--Yau type, that is, there exists an effective $\mathbb{Q}$-divisor $\Delta$ on $X$ such that $(X,\Delta)$ is log canonical and $K_{ X} +\Delta \sim_{\mathbb{Q}}0$. 
Letting $\Delta_Y:= f_*\Delta$, one has 
$$K_Y+\Delta_Y=f_*(K_X+\Delta) \sim_{\mathbb{Q}}0.$$ 
On the other hand, since $f$ is small, 
$$f^*(K_Y+\Delta_Y)= K_X +\Delta, $$
which implies that $(Y,\Delta_Y)$  is log canonical. 

Conversely, we assume that $Y$ is of Calabi--Yau type. Let $\Gamma$ be an effective $\mathbb{Q}$-divisor on $Y$ such that $(Y,\Gamma)$ is log canonical and $K_{Y} +\Gamma \sim_{\mathbb{Q}}0$. 
Let $\Gamma_X$ denote the strict transform of $\Gamma$ on $X$. Since $f$ is small, we see that 
$$K_X +\Gamma_X=f^*(K_Y+\Gamma),$$
which implies that $(X,\Gamma_X)$ is log canonical and $K_X+\Gamma_X \sim_{\mathbb{Q}}0$.
\end{proof}

\begin{proof}[Proof of Theorem \ref{logcy}]\label{proof of mthm}
First we remark that if $Y$ is a $\mathbb{Q}$-factorial Mori dream space of dense globally $F$-split type, then $-K_Y$ is $\mathbb{Q}$-linearly equivalent to an effective $\mathbb{Q}$-divisor on $Y$. Indeed, choosing a suitable integer $m \ge 1$, by Lemma \ref{good model}, we can take a model $Y_A$ of $Y$ over a finitely generated $\mathbb{Z}$-subalgebra $A$ and a dense subset $S \subseteq \Spec \, A$ such that 
\begin{enumerate}
\item $Y_{\mu}=Y_A \times_{\Spec \, A} \Spec \, k(\mu)$ is globally $F$-split,
\item $R(Y, -mK_Y)_{\mu} \cong R(Y_{\mu},  -mK_{Y_{\mu}})$
\end{enumerate}
for all closed points $\mu \in S$. 
Since $-K_{Y_{\mu}}$ is $\mathbb{Q}$-linearly equivalent to some effective $\mathbb{Q}$-divisor by (1) and Theorem \ref{ss1}, it follows from (2) that so is $-K_Y$. 

Since $X$ is a $\mathbb{Q}$-factorial Mori dream space, we can run a $(-K_X)$-MMP: 
$$X=X_0 \overset{f_0}{\dashrightarrow} X_1 \overset{f_1}{\dashrightarrow} \cdots  \overset{f_{l-2}}{\dashrightarrow} X_{l-1} \overset{f_{l-1}}{\dashrightarrow} X_l=X',$$ 
where each $X_i$ is a $\mathbb{Q}$-factorial Mori dream space and $X'$ is a $(-K_X)$-minimal model. 
Note that each $X_i$ is of dense globally $F$-split type by Lemma \ref{glFness for MMP}.  
In particular, $X'$ has only log canonical singularities by Lemma \ref{just singularities} (1). 
Then $X'$ is of Calabi--Yau type, because $-K_{X'}$ is semi-ample. 

Now we show that $X_{l-j}$ is of Calabi--Yau type by induction on $j$. 
When $j=0$, we have already seen that $X'=X_l$ is of Calabi--Yau type. 
Suppose that $X_{l-j+1}$ is of Calabi--Yau type. 
Let $\Delta_{l-j+1}$ be an effective $\mathbb{Q}$-divisor on $X_{l-j+1}$ such that $(X_{l-j+1},\Delta_{l-j+1})$ is log canonical and $K_{ X_{l-j+1}} +\Delta_{l-j+1} \sim_{\mathbb{Q}}0$. 

When $f:=f_{l-j}$ is a divisorial contraction, by an argument similar to the proof of Theorem \ref{logfano}, we have
$$-(K_{X_{l-j}}+\Delta_{l-j})= -f^*(K_{X_{l-j+1}}+ \Delta_{l-j+1})+a E,$$ 
where $\Delta_{l-j}$ is the strict transform of $\Delta_{l-j+1}$ on $X_{l-j}$, $E$ is the $f$-exceptional prime divisor on $X_{l-j}$ and $a$ is a positive rational number. 
It then follows that  $(X_{l-j}, \Delta_{l-j}+aE)$ is log canonical 
and $K_{X_{l-j}}+\Delta_{l-j}+aE \sim_{\mathbb{Q}}0$, 
that is,  $X_{l-j}$ is of Calabi--Yau type.

When $f_{l-j}$ is a $(-K_{X_{l-j}})$-flip, by an argument similar to the proof of Theorem \ref{logfano}, Lemma \ref{inv-flip-cy} implies that $X_{l-j}$ is of Calabi--Yau type. 

Thus, we conclude that $X=X_0$ is of Calabi--Yau type. 
  \end{proof}
\section{Characterization of varieties of Fano type}\label{cox ring}
In this section, we give a characterization of varieties of Fano type in terms of the singularities of their Cox rings. 

\begin{lem}\label{lem-ring-sum}
Let $X$ be a normal projective variety over a field $k$ of characteristic zero.
Let $\Gamma$ be a finitely generated semigroup of Weil divisors on $X$ and
$\Gamma'\subset\Gamma$ be a sub-semigroup.
If $R_{X}(\Gamma)$ is of strongly $F$-regular type, so is $R_X(\Gamma')$, provided that
both of them are of finite type over $k$.
\end{lem}

\begin{proof}
Let $(X_A, \Gamma_A, \Gamma'_A)$ be a model of $(X, \Gamma, \Gamma')$ over a finitely generated $\mathbb{Z}$-subalgebra $A$ of $k$.

Note that the natural inclusion $\iota:R_{X_A}(\Gamma_{A}')\subset R_{X_A}(\Gamma_{A})$ splits.
In fact we have the natural $R_{X_A}(\Gamma_{A}')$ module homomorphism
\begin{equation*}
\varphi:R_{X_A}(\Gamma_{A})\to R_{X_A}(\Gamma_{A}'),
\end{equation*}
which is defined as follows:
for $f\in R_{X_A}(\Gamma_{A})$, write $f=\sum_{D_{A}\in\Gamma_{A}}f_{D_{A}}$. Define
\begin{equation*}
\varphi(f)=\sum_{D_{A}\in\Gamma_{A}'}f_{D_{A}}.
\end{equation*}
It is easy to see that $\varphi$ is a $R_{X_A}(\Gamma_{A}')$-linear and
$\varphi\circ\iota=\mathrm{id}_{R_{X_A}(\Gamma_{A}')}$.
Once we have such a splitting, it is clear that for any closed point $\mu\in\Spec \, A$,
$R_{X_A}(\Gamma_{A}')\otimes_Ak(\mu)$ is a split subring of $R_{X_A}(\Gamma_{A})
\otimes_Ak(\mu)$.
Now the conclusion follows from the fact that the strong $F$-regularity descends to
a direct summand (see \cite[Theorem 3.1]{hh}).
\end{proof}

\begin{defi}[{\cite[(2.1)]{Hashi}}]
Let $\Gamma$ be a finitely generated torsion free abelian group. 
Let $R$ be a (not necessarily Noetherian) $\Gamma$-graded integral domain of characteristic $p>0$. 
For each integer $e \ge 1$, $F^e_*R$ is just $R$ as an abelian group, but its $R$-module structure is determined by $r \cdot x:= r^{p^e}x$ for all $r \in R$ and $x \in F^e_*R$. 
We give $F^e_*R$ a $\frac{1}{p^e}\Gamma$-module structure by putting $[F^e_*R]_{n/p^e}=[R_n]$. 

We say that $R$ is \textit{quasi-$F$-regular} if for any homogeneous nonzero element $c \in R$ of degree $n$, there exists an integer $e \ge 1$ such that 
$$cF^e: R \to F^e_*R(n)$$
splits as a $\frac{1}{p^e}\Gamma$-graded $R$-linear map, where $R(n)$ denotes the degree shifting of $R$ by $n$.
\end{defi}

\begin{rem}
When $R$ is a Noetherian $F$-finite $\Gamma$-graded integral domain, 
$R$ is quasi-$F$-regular if and only if $R$ is strongly $F$-regular. 
\end{rem}

\begin{rem}
The notion of $F$-purity can be defined for non-Noetherian rings. 
Let $R$ be a (not necessarily Noetherian) ring of prime characteristic $p$. 
We say $R$ is $F$-pure if the Frobenius map $R \to F_*R$ is pure, that is, $M \to F_*R \otimes_R M$ is injective for every $R$-module $M$. 
When $R$ is a Noetherian and $F$-finite, this definition coincides with that given in Definition \ref{defi-sFreg}. 
\end{rem}

\begin{lem}[{cf.~\cite[Lemma 2.10]{Hashi}}]\label{sannai lem} 
Let $X$ be a normal projective variety defined over an $F$-finite field of characteristic $p>0$ and $\Gamma$ be a semigroup of Weil divisors on $X$. 
If $X$ is globally $F$-regular $($resp. globally $F$-split$)$, then $R_X(\Gamma)$ is quasi-$F$-regular $($resp. $F$-pure$)$. 
\end{lem}
\begin{proof}
The globally $F$-regular case follows from \cite[Lemma 2.10]{Hashi} and the globally $F$-split case also follows from essentially the same argument. 
\end{proof}

\begin{prop}\label{sannai thm}
Let $X$ be a normal projective variety over an $F$-finite field of characteristic $p>0$. 
Then $X$ is globally $F$-regular if and only if its Cox rings are quasi-$F$-regular.
\end{prop}

\begin{proof}
If $X$ is globally $F$-regular, then by Lemma \ref{sannai lem}, any multisection ring of $X$ is quasi-$F$-regular, and so are the Cox rings of $X$.\\
Conversely, suppose that a cox ring $R_X(\Gamma)$ is quasi-$F$-regular. 
Since $\Gamma$ contains an ample divisor $H$ on $X$, its section ring $R(X, H)$ is a graded direct summand of $R_X(\Gamma)$. 
Since $R(X, H)$ is Noetherian and $F$-finite, 
this implies that $R(X, H)$ is strongly $F$-regular. 
It then follows from Proposition \ref{glsec} that $X$ is globally $F$-regular. 
\end{proof}

\begin{thm}\label{sannai conj}
Let $X$ be a $\mathbb{Q}$-factorial
projective variety over an algebraically closed field $k$ of characteristic zero. Then $X$ is of Fano type if and only if it is a Mori dream space and its Cox rings have
only log terminal singularities.
\end{thm}

\begin{proof}[Proof of Theorem \ref{sannai conj}]
Let $\Gamma$ be a group of Cartier divisors on $X$ which defines a Cox ring of $X$.

First assume that $X$ is of Fano type.
Then by \cite[Corollary 1.3.2]{bchm}, $R_X(\Gamma)$ is a finitely generated algebra over $k$. 
Also by Proposition \ref{just singularities}, $X$ is of globally $F$-regular type. 
Replacing $\Gamma$ with a suitable positive multiple if necessary, by Lemma \ref{good model}, we can take a model $(X_A, \Gamma_A)$ of $(X, \Gamma)$ over a finitely generated $\mathbb{Z}$-subalgebra $A$ of $k$ such that 
\begin{enumerate}
\item $X_{\mu}=X_A\times_{\Spec \, A }{\Spec \, k(\mu)}$ is globally $F$-regular,
\item $R_X(\Gamma)_{\mu}=R_{X_A}(\Gamma_A) \otimes_A k(\mu) \cong R_{X_{\mu}}(\Gamma_{\mu})$
\end{enumerate}
for all closed points $\mu \in \Spec \, A$. 

It follows from Lemma \ref{sannai lem} and (1) that $R_{X_\mu}(\Gamma_{\mu})$ is strongly $F$-regular for all closed points $\mu \in \Spec \, A$, which means by (2) that $R_X(\Gamma)$ is of strongly $F$-regular type.  
Since $\Spec {R_X(\Gamma)}$ is $\mathbb{Q}$-Gorenstein, we can conclude from \cite[Theorem 3.9]{hw} that $\Spec \, R_X(\Gamma)$ has only  log terminal singularities. 

 Conversely, suppose that the Cox ring $R_X{(\Gamma)}$ of $X$ is finitely generated over $k$ and has only log terminal singularities. 
Then we see that $R_X{(\Gamma)}$ is of strongly $F$-regular type by \cite[Theorem 5.2]{hara}.
Take an ample divisor $H\in\Gamma$ on $X$. 
Since $R(X, H)=R_X(\mathbb{Z}H)$ and $\mathbb{Z}H$ is a sub-semigroup of $\Gamma$, 
by Lemma \ref{lem-ring-sum}, $R(X, H)$ is also of strongly $F$-regular type.
By replacing $H$ with its positive multiple and enlarging $A$ if necessary,
we may assume that
$$ R(X_A,H_A)\otimes_A k(\mu)\cong R(X_{\mu},H_{\mu})$$
holds for any closed point $\mu\in\Spec{A}$ (use the Serre vanishing theorem
and the Grauert theorem \cite[Corollary 12.9]{har}).
It then follows from Proposition \ref{glsec} that $X$ is of globally $F$-regular type, which implies  
by Theorem \ref{logfano} that $X$ is of Fano type. 
Thus, we finish the proof of Theorem \ref{sannai conj}.
\end{proof}

\begin{rem}
Brown \cite{brown-Fano} proved without using characteristic $p$ methods a special case of Theorem \ref{sannai conj} that if $X$ is a $\mathbb{Q}$-factorial Fano variety with only log terminal singularities, then its Cox rings have only log terminal singularities. 
We, however, don't know how to prove Theorem \ref{sannai conj} in its full generality without using characteristic $p$ methods. 
\end{rem}

\begin{rem}
Suppose that $X$ is a variety of Fano type defined over an algebraically closed field of characteristic zero. 
If $X$ is in addition locally factorial and $\Cl{(X)}$ is free, then its cox rings have only Gorenstein canonical singularities. 
\end{rem}

By an argument similar to the proof of Theorem \ref{sannai conj}, we can show that if a Calabi--Yau surface $X$ is a Mori dream space, then the Cox rings of $X$ have only log canonical singularities. 
\begin{thm}\label{CY surface}
Let $X$ be a klt projective surface over an algebraically closed field $k$ of characteristic zero such that $K_X \sim_{\mathbb{Q}} 0$. 
If $X$ is a Mori dream space, then its Cox rings have only log canonical singularities. 
\end{thm}

\begin{proof}
The proof is similar to that of Theorem \ref{sannai conj}. 
Let $H$ be any ample Cartier divisor on $X$. 
By \cite[Proposition 5.4]{schsmith-logfano}, the affine cone $\Spec \, R(X, H)$ of $X$ has only log canonical singularities and its vertex is an isolated non-klt point of $\Spec \, R(X, H)$. 
It then follows from \cite[Corollary 3.6]{ft} that $R(X, H)$ is of dense $F$-pure type, which implies by Proposition \ref{glsec} that $X$ is of dense globally $F$-split type. 

Let $\Gamma$ be a group of Cartier divisors on $X$ which defines a Cox ring of $X$.
Replacing $\Gamma$ with its positive multiple if necessary, by Lemma \ref{good model}, we can take a model $X_A$ of $X$ and $\Gamma_A$ of $\Gamma$ over a finitely generated $\mathbb{Z}$-subalgebra $A$ of $k$ and a dense subset $S \subseteq \Spec A$ of closed points such that 
\begin{enumerate}
\item $X_{\mu}$ \textup{ is globally $F$-split,}
\item $R_X(\Gamma)_{\mu}=R_{X_A}(\Gamma_A) \otimes_A k(\mu) \cong R_{X_{\mu}}(\Gamma_{\mu})$
\end{enumerate}
for all $\mu \in S$. 
It then follows from Lemma \ref{sannai lem}  and (1) that  $R_{X_{\mu}}(\Gamma_{\mu})$ is $F$-pure for all closed points $\mu \in S$, which means by (2) that $R_X(\Gamma)$ is of dense $F$-pure type. 
Since $R_X(\Gamma)$ is $\mathbb{Q}$-Gorenstein, we can conclude from \cite[Theorem 3.9]{hw} that $\Spec \, R_X(\Gamma)$ has only log canonical singularities. 
\end{proof}

\begin{rem}
The notion of $F$-purity is defined also for a pair of a normal variety $X$ and an effective $\mathbb{Q}$-divisor $\Delta$ on $X$ (the reader is referred to \cite[Definition 2.1]{hw} for the definition of $F$-pure pairs). 
It is conjectured that modulo $p$ reduction of a log canonical pair is $F$-pure for infinitely many $p$ :

\begin{conj}[{cf.~\cite[Problem 5.1.2]{hw}}]\label{Fpure vs lc}
Let $X$ be a normal variety over an algebraically closed field of characteristic zero and $\Delta$ be an effective $\mathbb{Q}$-divisor on $X$ such that $K_X+\Delta$ is $\mathbb{Q}$-Cartier. 
Then the pair $(X, \Delta)$ is log canonical if and only if it is of dense $F$-pure type.
\end{conj}

If Conjecture \ref{Fpure vs lc} is true, then we can give a characterization of Mori dream spaces of Calabi--Yau type in terms of the singularities of their Cox rings, using an argument similar to the proof of Theorem \ref{sannai conj}.

\begin{thm}
Let $X$ be a $\mathbb{Q}$-factorial Mori dream space over an algebraically closed field of characteristic zero. 
Suppose that Conjecture \ref{Fpure vs lc} is true. 
Then $X$ is of Calabi--Yau type if and only if its Cox rings have only log canonical singularities. 
\end{thm}

\end{rem}

\section{Case of Non-$\mathbb{Q}$-factorial Mori dream space}\label{non-Q-factorial}
In this section we generalize our results to not-necessarily $\mathbb{Q}$-factorial
Mori dream spaces (see \cite[Section 10]{o}).
These varieties admit a small $\mathbb{Q}$-factorial modification
by a Mori dream space, so that we can apply our results obtained
so far.

\begin{defi}
Let $X$ be normal projective variety whose divisor class group $\Cl{(X)}$ is
finitely generated.
Choose a finitely generated group of Weil divisors $\Gamma$ on $X$ such that
the natural map
$$ \Gamma_{\mathbb{Q}}\to\Cl{(X)}_{\mathbb{Q}}$$
is an isomorphism.
$X$ is said to be a \textit{not-necessarily $\mathbb{Q}$-factorial Mori dream space}
if the multi-section ring $R_X(\Gamma)$ is of finite type over the base field.
\end{defi}

When $X$ is $\mathbb{Q}$-factorial, this coincides with an ordinary Mori dream space.
The following is quite useful:
\begin{prop}\label{Q-factorization}
For a not necessarily $\mathbb{Q}$-factorial Mori dream space $X$
we can find a small birational morphism $X'\to X$ from a $\mathbb{Q}$-factorial
Mori dream space $X'$.
\end{prop}
\begin{proof}
This is essentially proven in \cite[Proof of Theorem 2.3]{ahl}. See also
\cite[Remark 10.3]{o}.
\end{proof}

\begin{cor}[not necessarily $\mathbb{Q}$-factorial version of Theorem \ref{sannai conj}]\label{non-Q-factorial version}
Let $X$ be a normal projective variety over a field $k$ of characteristic zero.
Then $X$ is of Fano type if and only if it is (not necessarily $\mathbb{Q}$-factorial)
Mori dream space and its Cox rings have only log terminal singularities.
\end{cor}

\begin{proof}
Suppose that $X$ is of Fano type. Then we can take a small $\mathbb{Q}$-factorization
$f:\tilde{X}\to X$ (see \cite[Corollary 1.4.3]{bchm}) and show that $\tilde{X}$ also is of Fano type by Lemma \ref{inv-flip}.
By Theorem \ref{sannai conj},
we see that $\tilde{X}$ is a Mori dream space and its Cox rings have only log terminal singularities.
Since $f$ is small, we see that $X$ is a not necessarily $\mathbb{Q}$-factorial
Mori dream space, and its Cox rings are the same as those of $\tilde{X}$.

Conversely, suppose that $X$ is a not necessarily $\mathbb{Q}$-factorial
Mori dream space and its Cox rings have only log terminal singularities. 
Take a small $\mathbb{Q}$-factorization
$f:\tilde{X}\to X$ as in Proposition \ref{Q-factorization}. Again by Theorem \ref{sannai conj}
we see that $\tilde{X}$ is of Fano type. Applying the arguments in the proof of
Lemma \ref{inv-flip}, we see that $X$ is of Fano type.
\end{proof}

\begin{cor}[not necessarily $\mathbb{Q}$-factorial version of Theorem \ref{logfano}]\label{logfano'}
Let $X$ be a not necessarily $\mathbb{Q}$-factorial Mori dream space over a field of characteristic zero.
Then $X$ is of Fano type if and only if it is of globally $F$-regular type. 
\end{cor}
\begin{proof}
Since both of the notions are preserved by taking a small $\mathbb{Q}$-factorization
and taking a small birational contraction, we can prove the equivalence by
taking the small $\mathbb{Q}$-factorization of $X$ as in Proposition \ref{Q-factorization}
and then apply Theorem \ref{logfano}.
\end{proof}

As an application of Corollary \ref{logfano'}, we show that the image of a variety of Fano type
again is of Fano type, which was first proven in \cite{fg2} (see also \cite{fg1}).

\begin{cor}\label{images of logFano} 
Let $f:X \to Y$ be a surjective morphism between normal
projective varieties over an algebraically closed field $k$ of characteristic zero. 
If $X$ is of Fano type, then $Y$ is of Fano type. 
\end{cor}

\begin{proof}
Taking the Stein factorization of $f$, we can assume either $f$ is an algebraic fiber space
or a finite morphism.
When $f$ is finite, it is dealt with in \cite{fg2}. Therefore we consider the case when $f$ is an algebraic
fiber space.
By \cite[Corollary 1.3.2]{bchm} and Theorem \ref{logfano'}, $X$ is a not necessarily $\mathbb{Q}$-factorial
Mori dream space and of globally $F$-regular type.
By \cite[Theorem 1.1]{o} and Lemma \ref{glFness for MMP}, we see that $Y$ is a
a not necessarily $\mathbb{Q}$-factorial Mori dream space
of globally $F$-regular type. Again by Theorem \ref{logfano}, we conclude that $Y$ is of Fano type.
\end{proof}

\begin{rem}\label{rem-algebraically closed} In this paper, the algebraic closedness of the ground field $k$ is used only where we use the results of \cite{bchm}.
\end{rem}

\end{document}